\documentclass[12pt,leqno]{article}
\usepackage{hyperref}
\usepackage{soul}
\usepackage[doc]{optional}
\usepackage{xcolor}
\definecolor{labelkey}{rgb}{0,0.08,0.45}
\definecolor{rekey}{rgb}{0,0.6,0.0}
\definecolor{Brown}{rgb}{0.45,0.0,0.05}
\usepackage{exscale,relsize}
\usepackage{amsmath}
\usepackage{amsfonts}
\usepackage{amssymb}
\usepackage{calc}
\usepackage{theorem}
\usepackage{pifont}      %needed by dingautolist
\usepackage{graphicx}
\newcommand{\scal}[2]{\langle{{#1},{#2}}\rangle}

\newcommand{\RR}{\ensuremath{\mathbb R}}

\newcommand{\RX}{\ensuremath{\,\left]-\infty,+\infty\right]}}
\newcommand{\RXX}{\ensuremath{\,\left[-\infty,+\infty\right]}}

\newcommand{\NN}{\ensuremath{\mathbb N}}

\newcommand{\menge}[2]{\big\{{#1} \mid {#2}\big\}}

\newcommand{\To}{\ensuremath{\rightrightarrows}}

\newcommand{\dom}{\ensuremath{\operatorname{dom}}}

\newcommand{\gra}{\ensuremath{\operatorname{gra}}}

\newcommand{\inte}{\ensuremath{\operatorname{int}}}

\newcommand{\ran}{\ensuremath{\operatorname{ran}}}

\renewcommand{\phi}{\ensuremath{\varphi}}

\newtheorem{theorem}{Theorem}[section]
\newtheorem{lemma}[theorem]{Lemma}
\newtheorem{fact}[theorem]{Fact}

\newtheorem{definition}[theorem]{Definition}

\theoremstyle{plain}{\theorembodyfont{\rmfamily}
}
\theoremstyle{plain}{\theorembodyfont{\rmfamily}
}
\theoremstyle{plain}{\theorembodyfont{\rmfamily}
}
\theoremstyle{plain}{\theorembodyfont{\rmfamily}
\newtheorem{example}[theorem]{Example}}
\theoremstyle{plain}{\theorembodyfont{\rmfamily}
\newtheorem{remark}[theorem]{Remark}}

\theoremstyle{plain}{\theorembodyfont{\rmfamily}
}

\begin{document}

%\sffamily

\title{\sffamily{For maximally monotone linear relations,
dense type, negative-infimum type, and Fitzpatrick-Phelps type
all coincide with monotonicity of the adjoint}}

\author{
Heinz H.\ Bauschke\thanks{Mathematics, Irving K.\ Barber School,
UBC Okanagan, Kelowna, British Columbia V1V 1V7, Canada. E-mail:
\texttt{heinz.bauschke@ubc.ca}.},\;
Jonathan M. Borwein\thanks{CARMA, University of Newcastle, Newcastle, New South Wales 2308, Australia. E-mail:
\texttt{jonathan.borwein@newcastle.edu.au}.},\;
 Xianfu
Wang\thanks{Mathematics, Irving K.\ Barber School, UBC Okanagan,
Kelowna, British Columbia V1V 1V7, Canada. E-mail:
\texttt{shawn.wang@ubc.ca}.},\; and Liangjin\
Yao\thanks{Mathematics, Irving K.\ Barber School, UBC Okanagan,
Kelowna, British Columbia V1V 1V7, Canada.
E-mail:  \texttt{ljinyao@interchange.ubc.ca}.}}

\date{March 30, 2011}
\maketitle

\begin{abstract} \noindent
It is shown that, for maximally monotone linear relations defined on
a general Banach space, the monotonicities of dense type, of
negative-infimum type, and of Fitzpatrick-Phelps type are the same
and equivalent to monotonicity of the adjoint. This result also
provides affirmative answers to two problems: one posed by Phelps
and Simons, and the other by  Simons.
\end{abstract}

\noindent {\bfseries 2010 Mathematics Subject Classification:}\\
{Primary  47A06, 47H05;
Secondary
47B65, 47N10,
 90C25}

\noindent {\bfseries Keywords:}
Adjoint,
linear relation,
Fenchel conjugate,
maximally monotone operator,
monotone operator,
  operators of type (D),
 operators of type (FP),
operators of type (NI),
set-valued operator.

\section{Introduction}

Throughout this paper, we assume that
$X$ is a real Banach space with norm $\|\cdot\|$,
that $X^*$ is the continuous dual of $X$, and
that $X$ and $X^*$ are paired by $\scal{\cdot}{\cdot}$.
Let $A\colon X\To X^*$
be a \emph{set-valued operator} (also known as multifunction)
from $X$ to $X^*$, i.e., for every $x\in X$, $Ax\subseteq X^*$,
and let
$\gra A = \menge{(x,x^*)\in X\times X^*}{x^*\in Ax}$ be
the \emph{graph} of $A$. The \emph{domain} of $A$, written as
 $\dom A$,  is $\dom A= \menge{x\in X}{Ax\neq\varnothing}$ and
$\ran A=A(X)$ for the \emph{range} of $A$.
Recall that $A$ is  \emph{monotone} if
\begin{equation}
\scal{x-y}{x^*-y^*}\geq 0,\quad \forall (x,x^*)\in \gra A\;\forall (y,y^*)\in\gra A,
\end{equation}
and \emph{maximally monotone} if $A$ is monotone and $A$ has
 no proper monotone extension
(in the sense of graph inclusion).
Let $A:X\rightrightarrows X^*$ be monotone and $(x,x^*)\in X\times X^*$.
 We say $(x,x^*)$ is \emph{monotonically related to}
$\gra A$ if
\begin{align*}
\langle x-y,x^*-y^*\rangle\geq0,\quad \forall (y,y^*)\in\gra A.\end{align*}
 We now define the three aforementioned types of maximally monotone operators.

 \begin{definition}
 Let $A:X\To X^*$ be maximally monotone.
 Then three key types of monotone operators are defined as follows.
 \begin{enumerate}
 \item $A$ is
\emph{of dense type or type (D)}
(see \cite{Gossez3}) if for every $(x^{**},x^*)\in X^{**}\times X^*$ with
\begin{align*}
\inf_{(a,a^*)\in\gra A}\langle a-x^{**}, a^*-x^*\rangle\geq 0,
\end{align*}
there exist a  bounded net
$(a_{\alpha}, a^*_{\alpha})_{\alpha\in\Gamma}$ in $\gra A$
such that
$(a_{\alpha}, a^*_{\alpha})$ weak*$\times$strong converges to
$(x^{**},x^*)$.
\item $A$ is
\emph{of type negative infimum (NI)} (see \cite{SiNI}) if
\begin{align*}
\sup_{(a,a^*)\in\gra A}\big(\langle a,x^*\rangle+\langle a^*,x^{**}\rangle
-\langle a,a^*\rangle\big)
\geq\langle x^{**},x^*\rangle,
\quad \forall(x^{**},x^*)\in X^{**}\times X^*.
\end{align*}
\item
$A$ is
\emph{of type Fitzpatrick-Phelps (FP)}
(see \cite{FP})  if  for every open convex subset
$U$ of $X^*$ such that
$U\cap \ran A\neq\varnothing$, the implication
\begin{equation*}
x^*\in U\,\text{and}\,(x,x^*)\in X\times X^*\,\text{is monotonically related to $\gra A\cap (X\times U)$}
\Rightarrow (x,x^*)\in\gra A
\end{equation*}
holds.
\end{enumerate}
\end{definition}

We say $A$ is a \emph{linear relation} if $\gra A$ is a linear
subspace. By saying $A:X\To X^*$ is \emph{at most single-valued}, we
mean that for every $x\in X$, $Ax$ is either a singleton or empty.
In this case, we follow a slight but common abuse of notation and
write $A \colon \dom A\to X^*$. Conversely, if $T\colon D\to X^*$,
we may identify $T$ with $A:X\To X^*$, where $A$ is at most
single-valued with $\dom A = D$.

Monotone operators have proven to be a key class of objects in both
modern Optimization and Analysis; see, e.g., \cite{Bor1,Bor2,Bor3},
the books \cite{BC2011, BorVan,BurIus,ph,Si,Si2,RockWets,Zalinescu}
and the references therein.

 In this paper,
 we provide tools to give affirmative answers to two questions respectively
 posed by Phelps and Simons, and by Simons.
 Phelps and Simons posed the following question in
 \cite[Section~9, item~2]{PheSim}:
\emph{Let $A:\dom A\to X^*$ be linear and maximally monotone.
 Assume that $A^*$ is monotone. Is $A$ necessarily
 of type (D)}?
 
 Simons  posed another question in  \cite[Problem~47.6]{Si2}:
 \emph{Let $A:\dom A\rightarrow X^*$ be linear and maximally monotone.
 Assume that $A$ is of type (FP).  Is $A$ necessarily
 of type (NI)?}

We give affirmative answers to the above questions in
Theorem~\ref{TypeD:1}.
 Moreover,  we generalize the results  to the linear relations.
 Linear relations have recently become a center of attention
  in Monotone Operator Theory; see, e.g.,
\cite{BB,BBW,BWY2,BWY3,BWY4, BWY7, BWY8,BWY9,BuSv1,BuSv2, PheSim,
 Si3,Svaiter,Voisei06b,Voisei06,VZ,WY,Yao,Yao2}
and Cross' book \cite{Cross} for general background on linear
relations.

We adopt standard notation used in these books:
Given a subset $C$ of $X$,
$\inte C$ is the \emph{interior} of $C$,
and
$\overline{C}$ is   the \emph{norm closure} of $C$ .
The \emph{indicator function} of $C$, written as $\iota_C$, is defined
at $x\in X$ by
\begin{align}
\iota_C (x)=\begin{cases}0,\,&\text{if $x\in C$;}\\
+\infty,\,&\text{otherwise}.\end{cases}\end{align} For every $x\in
X$, the \emph{normal cone} operator of $C$ at $x$ is defined by
$N_C(x)= \menge{x^*\in X^*}{\sup_{c\in C}\scal{c-x}{x^*}\leq 0}$, if
$x\in C$; and $N_C(x)=\varnothing$, if $x\notin C$. For $x,y\in X$,
we set $\left[x,y\right]=\{tx+(1-t)y\mid 0\leq t\leq 1\}$. If $Z$ is
a real  Banach space with continuous dual $Z^*$ and a subset $S$ of
$Z$, we denote $S^\bot$ by $S^\bot = \menge{z^*\in Z^*}{\langle
z^*,s\rangle= 0,\quad \forall s\in S}$. Given a subset $D$ of $Z^*$,
we set $D_{\bot} = D^\perp \cap Z$. The \emph{adjoint} of $A$,
written as $A^*$, is defined by
\begin{equation*}
\gra A^* =
\menge{(x^{**},x^*)\in X^{**}\times X^*}{(x^*,-x^{**})\in(\gra A)^{\bot}}.
\end{equation*}
Let $f\colon X\to \RX$. Then
$\dom f= f^{-1}(\RR)$ is the \emph{domain} of $f$, and
$f^*\colon X^*\to\RXX\colon x^*\mapsto
\sup_{x\in X}(\scal{x}{x^*}-f(x))$ is
the \emph{Fenchel conjugate} of $f$.
For $\varepsilon \geq 0$,
the \emph{$\varepsilon$--subdifferential} of $f$ is defined by
   $\partial_{\varepsilon} f\colon X\To X^*\colon
   x\mapsto \menge{x^*\in X^*}{(\forall y\in
X)\; \scal{y-x}{x^*} + f(x)\leq f(y)+\varepsilon}$.
We also set $\partial f = \partial_{0}f$.

Let $F:X\times X^*\rightarrow\RX$. We say $F$ is  a \emph{representative}
 of a maximally monotone operator
$A:X\To X^*$ if $F$ is lower semicontinuous and convex with $F
 \geq\langle\cdot,\cdot\rangle$ on $X\times X^*$ and
\begin{align*}
\gra A=\{(x,x^*)\in X\times X^*\mid F(x,x^*)=\langle x,x^*\rangle\}.
\end{align*}
Let $(z,z^*)\in X\times X^*$. Then $F_{(z,z^*)}:X\times X^*\rightarrow\RX$
 \cite{MLegSva,Si2,MarSva2} is defined by
\begin{align}
F_{(z,z^*)}(x,x^*)&=F(z+x,z^*+x^*)-\big(\langle x,z^*\rangle+
\langle z,x^*\rangle+\langle z,z^*\rangle\big)\notag\\
&=F(z+x,z^*+x^*)-\langle z+x,z^*+x^*\rangle+\langle x,x^*\rangle,
\quad \forall(x,x^*)\in X\times X^*.\label{e:timmy}
\end{align}
Moreover, the \emph{closed unit ball} in $X$ is denoted by $B_X=
\menge{x\in X}{\|x\|\leq1}$, and $\NN=\{1,2,3,\ldots\}$. We identify
$X$ with its canonical image in the bidual space $X^{**}$.
Furthermore, $X\times X^*$ and $(X\times X^*)^* = X^*\times X^{**}$
are likewise paired via $$\scal{(x,x^*)}{(y^*,y^{**})} =
\scal{x}{y^*} + \scal{x^*}{y^{**}},$$ where $(x,x^*)\in X\times X^*$
and $(y^*,y^{**}) \in X^*\times X^{**}$. The norm on $X\times X^*$,
written as $\|\cdot\|_1$, is defined by
 $\|(x,x^*)\|_1=\|x\|+\|x^*\|$ for every $(x,x^*)\in X\times X^*$.

The remainder of this paper is organized as follows. In
Section~\ref{s:aux}, we collect auxiliary results for future
reference and for the reader's convenience. The main result
(Theorem~\ref{TypeD:1}) is provided in Section~\ref{s:main}. The
affirmative answers to Phelps-Simons' and Simons' questions are then
apparent.

 \section{Auxiliary results}\label{s:aux}

 \begin{fact}[Rockafellar] \label{f:F4}
\emph{(See \cite[Theorem~3(a)]{Rock66}, \cite[Corollary~10.3]{Si2}
or
{\cite[Theorem~2.8.7(iii)]{Zalinescu}}.)}
Let $f,g:  X\rightarrow\RX$ be proper convex functions.
Assume that there exists a point $x_0\in\dom f \cap \dom g$
such that $g$ is continuous at $x_0$. For every $x^*\in X^*$,
we have
 \begin{align*}(f+g)^*(x^*)=\displaystyle\min_{y^*\in X^*}\left[f^*(y^*)+g^*(x^*-y^*)\right].
 \end{align*}
\end{fact}

 \begin{fact}[Borwein] \label{f:FJB4}
\emph{(See \cite[Theorem~1]{Borwein2}
or \cite[Theorem~3.1.1]{Zalinescu}.)}
Let $f: X\rightarrow\RX$ be a proper lower semicontinuous and convex function.
Let $\varepsilon>0$ and $\beta\geq0$
(where $\tfrac{1}{0}=\infty$). Assume that $x_0\in\dom f$
 and $x^*_0\in\partial_{\varepsilon} f(x_0)$.
There exist  $x_{\varepsilon}\in X, x^*_{\varepsilon}\in X^*$ such that
\begin{align*}
&\|x_{\varepsilon}-x_0\|+\beta\left|\langle x_{\varepsilon}-x_0, x^*_0\rangle\right|
\leq \sqrt{\varepsilon}, \quad x^*_{\varepsilon}\in\partial
f(x_{\varepsilon}),\\
&\|x^*_{\varepsilon}-x^*_0\|\leq \sqrt{\varepsilon}(1+\beta\|x^*_0\|),
\quad \left|\langle x_{\varepsilon}
-x_0, x^*_{\varepsilon}\rangle\right|\leq\varepsilon+\frac{ \sqrt{\varepsilon}}{\beta}
.\end{align*}
\end{fact}

\begin{fact}[Simons]
\emph{(See \cite[Theorem~17]{Si4} or \cite[Theorem~37.1]{Si2}.)}
\label{FTSim:1}
Let $A:X\rightrightarrows X^*$ be a maximally monotone
operator such that $A$ is of type (D). Then $A$ is type
of (FP).
\end{fact}

\begin{fact}[Simons]
\emph{(See \cite[Lemma~19.7 and Section~22]{Si2}.)}
\label{f:referee}
Let $A:X\rightrightarrows X^*$ be a monotone operator such
 that $\operatorname{gra} A$ is convex with $\operatorname{gra} A
\neq\varnothing$.
Then the function
\begin{equation}
g\colon X\times X^* \rightarrow \left]-\infty,+\infty\right]\colon
(x,x^*)\mapsto \langle x, x^*\rangle + \iota_{\operatorname{gra} A}(x,x^*)
\end{equation}
is proper and convex.
\end{fact}

\begin{fact}[Marques Alves and Svaiter]\emph{(See \cite[Theorem~4.4]{MarSva}.)}
\label{PF:Su1} Let $A:X \To X^*$ be maximally  monotone, and
let  $F:X\rightarrow \RX$ be  a representative  of $A$.
Then the following are equivalent.
\begin{enumerate}
\item\label{MSF:1} $A$ is type of (D).
\item\label{MSF:2} $A$ is of type (NI).
\item \label{MSF:3}For every $(x_0,x^*_0)\in X\times X^*$,
\begin{align*}
\inf_{(x,x^*)\in X\times X^*}\left[F_{(x_0,x^*_0)}(x,x^*)+\tfrac{1}{2}\|x\|^2
+\tfrac{1}{2}\|x^*\|^2\right]=0.\end{align*}

\end{enumerate}
\end{fact}

\begin{remark}
The implication
\ref{MSF:1}$\Rightarrow$\ref{MSF:2} in Fact~\ref{PF:Su1}
 was first proved by Simons (see \cite[Lemma~15]{SiNI} or \cite[Theorem~36.3(a)]{Si2}).
\end{remark}

\begin{fact}[Cross]\label{Rea:1}
Let $A \colon X \To X^*$ be a linear relation.
Then the following hold.
\begin{enumerate}
\item \label{Th:28}$Ax=x^* +A0,\quad\forall x^*\in Ax.$
\item \label{Sia:2b}
$(\forall x^{**}\in \dom A^*)(\forall y\in\dom A)$
$\langle A^*x^{**},y\rangle=\langle x^{**}, Ay\rangle$ is a singleton.

\item \label{Sia:1} $(\dom A)^{\bot}=A^*0$.
 If $\gra A$ is  closed, then $(\dom A^*)_{\bot}=A0$.
\end{enumerate}
\end{fact}
\begin{proof}
\ref{Th:28}:  See \cite[Proposition I.2.8(a)]{Cross}.
\ref{Sia:2b}: See \cite[Proposition~III.1.2]{Cross}.
\ref{Sia:1} : See \cite[Proposition~III.1.4(b)\&(d)]{Cross}.
\end{proof}

\begin{lemma}\label{FCLL:1}
Let $A:X\To X^*$ be a maximally monotone linear relation. Then $(\dom A)^{\bot}=A0=A^*0=(\dom A^*)_{\bot}$.
\end{lemma}

\begin{proof}
(See also  \cite[Theorem~3.2(iii)]{BWY3} when $X$ is reflexive.)
Since $A+N_{\dom A}=A+(\dom A)^{\perp}$
is a monotone extension of $A$ and $A$ is
maximally monotone, we must have
$A+(\dom A)^{\perp}=A$. Then $A0+(\dom A)^{\perp}=A0$. As $0\in A0$,
$(\dom A)^{\perp}\subseteq A0.$

On the other hand,
take $x\in \dom A$.
Then there exists $x^*\in X^*$ such that
  $(x,x^*)\in \gra A.$  By
  monotonicity of $A$ and since $(0,A0)\subseteq \gra A$,
 we have  $\langle x, x^*\rangle\geq \sup\langle x,A0\rangle$.
Since $A0$ is a linear subspace,
we obtain
$x\bot A0$. This implies
$A0\subseteq (\dom A)^\bot$.

Combining the above, we have $(\dom A)^{\bot}=A0$.
Thus, by Fact~\ref{Rea:1}\ref{Sia:1}, $(\dom A)^{\bot}=A0=A^*0=(\dom A^*)_{\bot}$.
\end{proof}

\begin{lemma}\label{FCLL:2}
Let $A:X\To X^*$ be a maximally monotone linear relation. Then $\langle x^{**}, A^*x^{**}\rangle$
is single-valued for every $x^{**}\in\dom A^*$.
\end{lemma}

\begin{proof}Take $x^{**}\in\dom A^*$ and $x^*\in A^*x^{**}$.
By Fact~\ref{Rea:1}\ref{Th:28} and Lemma~\ref{FCLL:1},
\begin{align*}\langle x^{**},A^{*}x^{**}\rangle=\langle x^{**},x^*+A^*0\rangle=\langle x^{**},x^*\rangle.
\end{align*}
Thus $\langle x^{**}, A^*x^{**}\rangle$
is single-valued.
\end{proof}

\section{Main result}\label{s:main}
\begin{theorem}\label{TypeD:1}
Let $A:X\rightrightarrows X^*$ be a maximally monotone linear relation.
Then the following are equivalent.
\begin{enumerate}
\item\label{MaT:1} $A$ is of type (D).

\item \label{MaT:2} $A$ is of type (NI).
\item \label{MaT:3} $A^*$ is monotone.

\item \label{MaT:4} $A$ is of type (FP).
\end{enumerate}
\end{theorem}
\begin{proof}
``\ref{MaT:1}$\Leftrightarrow$\ref{MaT:2}'': Fact~\ref{PF:Su1}.

``\ref{MaT:2}$\Rightarrow$\ref{MaT:3}'':
Suppose to the contrary that there exists
 $(a^{**}_0, a^*_0)\in\gra A^*$ such that
$\langle a^{**}_0, a^*_0\rangle<0$.
Then we have
\begin{align*}
&\sup_{(a,a^*)\in\gra A}\big(\langle a,-a^*_0\rangle+\langle a^{**}_0,
 a^*\rangle-\langle a,a^*\rangle\big)=\sup_{(a,a^*)\in\gra A}\{-\langle a,a^*\rangle\}=0<\langle -a^{**}_0, a^*_0\rangle,
\end{align*}
which contradicts that $A$ is type of (NI).
Hence $A^*$ is monotone.

``\ref{MaT:3}$\Rightarrow$\ref{MaT:2}'':
Define
\begin{align*}
F:X\times X^*\to \RX\colon
(x,x^*)\mapsto \iota_{\gra A}(x,x^*)+\langle x,x^*\rangle.
\end{align*}
Since $A$ is maximally monotone,
Fact~\ref{f:referee} implies that
$F$ is proper lower semicontinuous and convex, and a representative of $A$.
Let $(v_0,v^*_0)\in X\times X^*$.
Recalling \eqref{e:timmy}, note that
\begin{align}
F_{(v_0,v^*_0)} \colon (x,x^*)
\mapsto \iota_{\gra A}(v_0+x,v^*_0+x^*)+\langle x,x^*\rangle \label{BrB:PSb1}
\end{align}
is proper lower semicontinuous and convex.
By  Fact~\ref{f:F4}, there exists
$(y^{**},y^*)\in X^{**}\times X^* $ such that
\begin{align}
K&:=\inf_{(x,x^*)\in X\times X^*}\left[F_{(v_0,v^*_0)}(x,x^*)+\tfrac{1}{2}\|x\|^2
+\tfrac{1}{2}\|x^*\|^2\right]\nonumber\\
&=-\big(F_{(v_0,v^*_0)}+\tfrac{1}{2}\|\cdot\|^2
+\tfrac{1}{2}\|\cdot\|^2\big)^*(0,0)\nonumber\\
&=-F^*_{(v_0,v^*_0)}(y^*,y^{**})-
\tfrac{1}{2}\|y^{**}\|^2-\tfrac{1}{2}\|y^*\|^2.
\label{BrB:PS1}
\end{align}
Since $(x,x^*)\mapsto F_{(v_0,v^*_0)}(x,x^*) + \tfrac{1}{2}\|x\|^2 +
\tfrac{1}{2}\|x^*\|^2$ is coercive, there exist $M> 0$ and a sequence
$(a_n,a^*_n)_{n\in\NN}$ in $X\times X^*$
such that
   \begin{align}\|a_n\|+\|a^*_n\|\leq M\label{coer:1}\end{align}
   and
\begin{align}
&F_{(v_0,v^*_0)}(a_n,a_n^*)+\tfrac{1}{2}\|a_n\|^2
+\tfrac{1}{2}\|a_n^*\|^2\nonumber\\
&<K+\tfrac{1}{n^2}=-F^*_{(v_0,v^*_0)}(y^*,y^{**})-\tfrac{1}{2}\|y^{**}\|^2-\tfrac{1}{2}\|y^*\|^2+
\tfrac{1}{n^2}
\quad\text{(by \eqref {BrB:PS1}
)}\nonumber
\\
&\Rightarrow F_{(v_0,v^*_0)}(a_n,a_n^*)+\tfrac{1}{2}\|a_n\|^2
+\tfrac{1}{2}\|a_n^*\|^2
+F^*_{(v_0,v^*_0)}(y^*,y^{**})+\tfrac{1}{2}\|y^{**}\|^2+\tfrac{1}{2}\|y^*\|^2
<\tfrac{1}{n^2}\label{BrB:PSa:01}\\
&\Rightarrow F_{(v_0,v^*_0)}(a_n,a_n^*)+F^*_{(v_0,v^*_0)}(y^*,y^{**})+
\langle a_n,-y^*\rangle+\langle a^*_n,-y^{**}\rangle<\tfrac{1}{n^2}\label{BrB:PS03}\\
&\Rightarrow (y^*,y^{**})
\in\partial_{\tfrac{1}{n^2}} F_{(v_0,v^*_0)}(a_n,a_n^*)\quad\text{(by
 \cite[Theorem~2.4.2(ii)]{Zalinescu})}.
\label{BrB:PS2}
\end{align}  Set $\beta=\frac{1}{\max\{\|y^*\|, \|y^{**}\|\}+1}$. Then
by  Fact~\ref{f:FJB4}, there exist sequences  $(\widetilde{a_n},
 \widetilde{a^{*}_n})_{n\in\NN}$ in $X\times X^*$   and
$(y^*_n, y^{**}_n)_{n\in\NN}$ in $X^{*}\times X^{**}$  such  that
 \begin{align}
&\|a_n-\widetilde{a_n}\|+ \|a_n^*-\widetilde{a^*_n}\|+\beta\left|\langle \widetilde{a_n}-a_n, y^*\rangle+
\langle \widetilde{a^*_n}-a^*_n, y^{**}\rangle\right|\leq\tfrac{1}{n}\label{BrB:PSa1}\\
&\max\{\|y^*_n-y^{*}\|,\|y^{**}_n-y^{**}\|\}\leq \tfrac{2}{n}\label {BrB:PSaa2}\\
&\left|\langle \widetilde{a_n}-a_n, y_n^*\rangle+
\langle \widetilde{a^*_n}-a^*_n, y_n^{**}\rangle\right|\leq\tfrac{1}{n^2}+\tfrac{1}{n\beta}\label{BrB:PSa2}\\
&(y^*_n, y^{**}_n)\in\partial F_{(v_0,v^*_0)}(\widetilde{a_n}, \widetilde{a^{*}_n})\label{BrB:PSaa3},\quad \forall n\in\NN.
\end{align}
Then we have
\begin{align}
&\langle \widetilde{a_n}, y^*_n\rangle+\langle \widetilde{a^*_n}, y^{**}_n\rangle
-\langle a_n, y^*\rangle-\langle a^*_n, y^{**}\rangle\nonumber\\
&=\langle \widetilde{a_n}-a_n, y^*_n\rangle+\langle a_n, y^*_n-y^*\rangle
+\langle \widetilde{a^*_n}-a^*_n, y^{**}_n\rangle+\langle a^*_n,y^{**}_n- y^{**}\rangle\nonumber\\
&\leq \left| \langle \widetilde{a_n}-a_n, y^*_n\rangle+\langle \widetilde{a^*_n}
-a^*_n, y^{**}_n\rangle\right|+\left|\langle a_n, y^*_n-y^*\rangle\right|+\left|\langle a^*_n,y^{**}_n
- y^{**}\rangle\right|\nonumber\\
&\leq\tfrac{1}{n^2}+\tfrac{1}{n\beta}+\|a_n\|\cdot\|y^*_n-y^*\|+\|a^*_n\|\cdot\|y^{**}_n
- y^{**}\|\quad \text{(by \eqref{BrB:PSa2})}\nonumber\\
&\leq\tfrac{1}{n^2}+\tfrac{1}{n\beta}+(\|a_n\|+\|a^*_n\|)\cdot\max\{\|y^*_n-y^*\|,\|y^{**}_n
- y^{**}\|\}\nonumber\\
&\leq\tfrac{1}{n^2}+\tfrac{1}{n\beta}+
\tfrac{2}{n}M\quad \text{(by
 \eqref{coer:1} and \eqref{BrB:PSaa2})},\quad \forall n\in\NN\label{BrB:PS9}.
\end{align}
By \eqref{BrB:PSa1}, we have
\begin{align}\big|\|a_n\|-\|\widetilde{a_n}\|\big|
+ \big|\|a_n^*\|-\|\widetilde{a^*_n}\|\big|\leq\tfrac{1}{n}.\label{BrB:De1}
\end{align}
Thus by \eqref{coer:1}, we have
\begin{align}&\left|\|a_n\|^2-\|\widetilde{a_n}\|^2\right|+
 \left|\|a_n^*\|^2-\|\widetilde{a^*_n}\|^2\right|\nonumber\\
&=\big|\|a_n\|-\|\widetilde{a_n}\|\big|
\big(\|a_n\|+\|\widetilde{a_n}\|\big)+
 \left|\|a_n^*\|-\|\widetilde{a^*_n}\|\right|\left(\|a_n^*\|+\|\widetilde{a^*_n}\|\right)\nonumber\\
 &\leq\tfrac{1}{n}\left(2\|a_n\|+\tfrac{1}{n}\right)
 +\tfrac{1}{n}\left(2\|a_n^*\|+\tfrac{1}{n}\right)\quad \text{(by \eqref{BrB:De1})}\nonumber\\
  &\leq\tfrac{1}{n}(2M+\tfrac{2}{n})=\tfrac{2}{n}M+\tfrac{2}{n^2},\quad \forall n\in\NN.
 \label{BrB:PS0ab}
\end{align}
Similarly, by \eqref{BrB:PSaa2}, for all $n\in\NN$, we have
\begin{align}\left |\|y^*_n\|^2-\|y^*\|^2\right|
\leq\tfrac{4}{n}\|y^*\|+\tfrac{4}{n^2}\leq\tfrac{4}{n\beta}+\tfrac{4}{n^2},\quad
\left |\|y^{**}_n\|^2-\|y^{**}\|^2\right|\leq
\tfrac{4}{n}\|y^{**}\|+\tfrac{4}{n^2}\leq\tfrac{4}{n\beta}+\tfrac{4}{n^2}.
 \label{BrB:PS0ab0}
\end{align}
Thus
\begin{align}&F_{(v_0,v^*_0)}(\widetilde{a_n},
\widetilde{a^{*}_n})+F^*_{(v_0,v^*_0)}(y^*_n, y^{**}_n)
+\tfrac{1}{2}\|\widetilde{a_n}\|^2
+\tfrac{1}{2}\|\widetilde{a^*_n}\|^2+\tfrac{1}{2}\|y^*_n\|^2
+\tfrac{1}{2}\|y^{**}_n\|^2\nonumber\\
&= \left[F_{(v_0,v^*_0)}(\widetilde{a_n},
\widetilde{a^{*}_n})+F^*_{(v_0,v^*_0)}(y^*_n, y^{**}_n)
+\tfrac{1}{2}\|\widetilde{a_n}\|^2
+\tfrac{1}{2}\|\widetilde{a^*_n}\|^2+\tfrac{1}{2}\|y^*_n\|^2
+\tfrac{1}{2}\|y^{**}_n\|^2\right]\nonumber\\
&\quad-\left[F_{(v_0,v^*_0)}(a_n,a_n^*)+\tfrac{1}{2}\|a_n\|^2
+\tfrac{1}{2}\|a_n^*\|^2
+F^*_{(v_0,v^*_0)}(y^*,y^{**})+\tfrac{1}{2}\|y^{**}\|^2
+\tfrac{1}{2}\|y^*\|^2\right]\nonumber\\
&\quad +
\left[F_{(v_0,v^*_0)}(a_n,a_n^*)+\tfrac{1}{2}\|a_n\|^2
+\tfrac{1}{2}\|a_n^*\|^2
+F^*_{(v_0,v^*_0)}(y^*,y^{**})+\tfrac{1}{2}\|y^{**}\|^2
+\tfrac{1}{2}\|y^*\|^2\right]\nonumber\\
&< \left[F_{(v_0,v^*_0)}(\widetilde{a_n},
\widetilde{a^{*}_n})+
F^*_{(v_0,v^*_0)}(y^*_n, y^{**}_n)-F_{(v_0,v^*_0)}(a_n,a_n^*)
-F^*_{(v_0,v^*_0)}(y^*,y^{**})\right]\nonumber\\
&\quad+\tfrac{1}{2}\left[\|\widetilde{a_n}\|^2
+\|\widetilde{a^*_n}\|^2-\|a_n\|^2
-\|a_n^*\|^2\right]\nonumber\\
&\quad+\tfrac{1}{2}\left[\|y^*_n\|^2
+\|y^{**}_n\|^2-\|y^{**}\|^2-\|y^*\|^2\right]
+\tfrac{1}{n^2}\quad \text{(by \eqref{BrB:PSa:01})}\nonumber\\
&\leq \left[\langle \widetilde{a_n}, y^*_n\rangle+\langle \widetilde{a^*_n}, y^{**}_n\rangle
-\langle a_n, y^*\rangle-\langle a^*_n, y^{**}\rangle\right]\quad \text{(by \eqref{BrB:PSaa3})}\nonumber\\
&\quad+\tfrac{1}{2}\left(\left|\|\widetilde{a_n}\|^2-\|a_n\|^2\right|
+\left|\|\widetilde{a^*_n}\|^2
-\|a_n^*\|^2\right|\right)\nonumber\\
&\quad+\tfrac{1}{2}\left(\left|\|y^*_n\|^2-\|y^*\|^2\right|
+\left|\|y^{**}_n\|^2-\|y^{**}\|^2\right|\right)
+\tfrac{1}{n^2}\nonumber\\
&\leq \tfrac{1}{n^2}+\tfrac{1}{n\beta}+\tfrac{2}{n}M+\tfrac{1}{n}M+\tfrac{1}{n^2}
+\tfrac{4}{n\beta}
+\tfrac{4}{n^2}+\tfrac{1}{n^2}\quad \text{(by \eqref{BrB:PS9},  \eqref{BrB:PS0ab} and \eqref{BrB:PS0ab0})}\nonumber\\
&=\tfrac{7}{n^2}+\tfrac{5}{n\beta}+\tfrac{3}{n}M,\quad \forall n\in\NN.
\label{BrB:PS4}
\end{align}
By \eqref{BrB:PSaa3}, \eqref{BrB:PSb1},
and \cite[Theorem 3.2.4(vi)\&(ii)]{Zalinescu},
 there exists a sequence
 $(z^*_n,z^{**}_n)_{n\in\NN}$ in
$(\gra A)^{\bot}$ and such that
\begin{align}
(y^*_n, y^{**}_n)=(\widetilde{a^{*}_n}, \widetilde{a_n})+(z^*_n,z^{**}_n),\quad \forall n\in\NN.\label{BrB:PS10}
\end{align}
Since $A^*$ is monotone and $(z^{**}_n,z^{*}_n)\in\gra(- A^*)$,
it follows from \eqref{BrB:PS10} that
\begin{align}&\langle y^*_n, y^{**}_n\rangle-\langle y_n^*,
 \widetilde{a_n}\rangle-\langle y_n^{**}, \widetilde{a^{*}_n}\rangle
 +\langle \widetilde{a^{*}_n},\widetilde{a_n}\rangle =
\langle y^*_n- \widetilde{a^{*}_n}, y^{**}_n-\widetilde{a_n}\rangle
=\langle z^*_n,z^{**}_n\rangle\leq0\nonumber\\
&\Rightarrow \langle y^*_n, y^{**}_n\rangle\leq \langle y_n^*,
 \widetilde{a_n}\rangle+\langle y_n^{**}, \widetilde{a^{*}_n}\rangle
 -\langle \widetilde{a^{*}_n},\widetilde{a_n}\rangle,\quad \forall n\in\NN.\nonumber\end{align}
Then by
\eqref{BrB:PSb1} and \eqref{BrB:PSaa3}, we have
$\langle \widetilde{a^{*}_n},\widetilde{a_n}\rangle = F_{(v_0,v^*_0)}(\widetilde{a_n},
\widetilde{a^{*}_n})$ and
 \begin{align}
&\langle y^*_n, y^{**}_n\rangle\leq \langle y_n^*,
 \widetilde{a_n}\rangle+\langle y_n^{**}, \widetilde{a^{*}_n}\rangle-F_{(v_0,v^*_0)}(\widetilde{a_n},
\widetilde{a^{*}_n})=F^*_{(v_0,v^*_0)}(y^*_n, y^{**}_n),\quad \forall n\in\NN.\label{BrB:PS11}
\end{align}
By \eqref{BrB:PS4} and \eqref{BrB:PS11}, we have
\begin{align}
&F_{(v_0,v^*_0)}(\widetilde{a_n},
\widetilde{a^{*}_n})+\langle y^*_n, y^{**}_n\rangle
+\tfrac{1}{2}\|\widetilde{a_n}\|^2
+\tfrac{1}{2}\|\widetilde{a^*_n}\|^2+\tfrac{1}{2}\|y^*_n\|^2
+\tfrac{1}{2}\|y^{**}_n\|^2< \tfrac{7}{n^2}+\tfrac{5}{n\beta}+\tfrac{3}{n}M\nonumber\\
&\Rightarrow F_{(v_0,v^*_0)}(\widetilde{a_n},
\widetilde{a^{*}_n})
+\tfrac{1}{2}\|\widetilde{a_n}\|^2
+\tfrac{1}{2}\|\widetilde{a^*_n}\|^2<\tfrac{7}{n^2}+\tfrac{5}{n\beta}+\tfrac{3}{n}M,\quad \forall n\in\NN.\label{BrB:PS12}
\end{align}
Thus by \eqref{BrB:PS12},
\begin{align}
\inf_{(x,x^*)\in X\times X^*}\left[F_{(v_0,v^*_0)}(x,x^*)+\tfrac{1}{2}\|x\|^2
+\tfrac{1}{2}\|x^*\|^2\right]\leq 0.\label{BrB:PS13}
\end{align}
By \eqref{BrB:PSb1},
\begin{align}
\inf_{(x,x^*)\in X\times X^*}\left[F_{(v_0,v^*_0)}(x,x^*)+\tfrac{1}{2}\|x\|^2
+\tfrac{1}{2}\|x^*\|^2\right]\geq0.\label{BrB:PS14}
\end{align}
Combining \eqref{BrB:PS13} with \eqref{BrB:PS14}, we  obtain
\begin{align}
\inf_{(x,x^*)\in X\times X^*}\left[F_{(v_0,v^*_0)}(x,x^*)+\tfrac{1}{2}\|x\|^2
+\tfrac{1}{2}\|x^*\|^2\right]=0.
\end{align}
Thus by Fact~\ref{PF:Su1}, $A$ is of type (NI).  This concludes the
proof that \ref{MaT:1}, \ref{MaT:2}, and \ref{MaT:3} coincide.

Now ``\ref{MaT:1}$\Rightarrow$\ref{MaT:4}'' follows from
Fact~\ref{FTSim:1}. It remains to show only:

``\ref{MaT:4}$\Rightarrow$\ref{MaT:3}'':
Let $(x^{**}_0,x^*_0)\in\gra A^*$.
We must show that
\begin{align}
\langle x^{**}_0,x^*_0\rangle\geq 0.\label{FPCM:0a1}
\end{align}
We can and do  assume that
\begin{align}
\langle x^{**}_0,x^*_0\rangle\neq 0. \label{FPCM:01}
\end{align}
By Fact~\ref{Rea:1}\ref{Sia:2b},
\begin{align}
\langle x^{**}_0, Aa\rangle=\langle x^{*}_0, a\rangle,\quad \forall a\in\dom A.\label{FPCM:1}
\end{align}
We claim  that there exists $a_0\in\dom A$ such that
\begin{align}
\langle x^{*}_0, a_0\rangle<0.\label{FPCM:2}
\end{align}
Recalling that $\dom A$ is a subspace,
we suppose to the contrary that
\begin{align}
\langle x^{*}_0, a\rangle=0,\quad \forall a\in\dom A.
\end{align}
Thus
\begin{align}
(0,x^{*}_0)\in\gra A^*.
\end{align}
Since $(x^{**}_0,x^*_0)\in\gra A^*$,  $(x^{**}_0,0)\in\gra A^*$.
Thus, by Lemma~\ref{FCLL:2},
\begin{align}
\langle x^{**}_0,x^*_0\rangle=\langle x^{**}_0,0\rangle= 0,
\end{align}
which contradicts \eqref{FPCM:01}.
Hence \eqref{FPCM:2} holds.
Take $a^*_0\in X^*$ such that $(a_0,a^*_0)\in\gra A$. Set
\begin{align}
C_n=\left[a^*_0, x^*_0\right]+\tfrac{1}{n} B_{X^*}.\label{FPCM:03}
\end{align}
Then $C_n$ is weak$^*$ compact, convex,
and $x^*_0\in \inte C_n$.

Now we show that
\begin{align}
(0,x^*_0)\notin \gra A.\label{FPCM:3}
\end{align}
Suppose to the contrary that $(0, x^*_0)\in\gra A$. By Lemma~\ref{FCLL:1},
 $(0,x^*_0)\in\gra A^*$.
 Since $(x^{**}_0,x^*_0)\in\gra A^*$, $(x^{**}_0,0)\in\gra A^*$. Thus
  by Lemma~\ref{FCLL:2} again,
we have
\begin{align}
\langle x^{**}_0,x^*_0\rangle=\langle x^{**}_0,0\rangle= 0,
\end{align}
which contradicts \eqref{FPCM:01}.
Thus \eqref{FPCM:3} holds.

By \eqref{FPCM:03}, $x^*_0\in\inte C_n$.  Then by \eqref{FPCM:3}, $a^*_0\in\ran A\cap\inte C_n$ and that $A$ is of type (FP), we have
\begin{align}
0>&\inf_{(a,a^*)\in X\times X^*}\big(-\langle x^*_0, a\rangle+\langle a,a^*\rangle+
 \iota_{\gra A}(a,a^*)+ \iota_{X\times C_n}(a,a^*)\big)\notag\\
&=-\left[\langle \cdot,\cdot\rangle+
 \iota_{\gra A}+ \iota_{X\times C_n}\right]^*(x^*_0,0),
 \quad \forall n\in\NN. \label{FEPCM:b01}
\end{align}
By Fact~\ref{f:referee},
 \begin{equation}
F\colon X\times X^* \to \RX\colon
(x,x^*)\mapsto \scal{x}{x^*} + \iota_{\gra A}(x,x^*)\quad\text{ is proper and convex
}. \label{FEPCM:b1}
\end{equation}
Since
\begin{equation} (a_0,a^*_0)\in\gra A\quad\text{and}\quad
a^*_0 \in \ran A \cap \inte C_n, \quad \forall n\in\NN,
\end{equation}
$(a_0,a^*_0)\in\dom F\cap\inte \dom  \iota_{X\times C_n}$.
Then \begin{align}
 \iota_{X\times C_n}\, \text{is continuous at $(a_0,a^*_0)$},\quad \forall n\in\NN. \label{FEPCM:b2}
\end{align}
Using \eqref{FEPCM:b01}, \eqref{FEPCM:b2},  \eqref{FEPCM:b1},
Fact~\ref{f:F4}, and the fact that
$(x_0^{**},x_0^*)\in\gra A^* \Leftrightarrow F^*(x_0^*,-x_0^{**})=0$,
we have
 \begin{align}
 &0>-\min_{(y^{**},y^*)\in X^{**}\times X^*}\left[F^*(x^*_0+y^*,y^{**})+  \iota^*_{X\times C_n}(-y^*,-y^{**})\right]\nonumber\\
 &\geq-\left[F^*(x^*_0,-x^{**}_0)+  \iota^*_{X\times C_n}(0,x_0^{**})\right]\nonumber\\
 &=- \iota^*_{X\times C_n}(0,x_0^{**})\nonumber\\
 &=-\tfrac{1}{n}\|x_0^{**}\|-\max\{\langle x^*_0,x^{**}_0\rangle,\langle x^{**}_0, a^*_0\rangle\}.
  \label {FPCM:5}
\end{align}
Take $n\to \infty$ in \eqref{FPCM:5} to get
\begin{align}
\max\{\langle x^*_0,x^{**}_0\rangle,\langle x^{**}_0, a^*_0\rangle\}\geq 0.  \label {FPCM:7}\end{align}
Since \begin{align}
\langle x^{**}_0, a^*_0\rangle=\langle x^*_0,  a_0\rangle<0,
\quad\text{(by \eqref{FPCM:1} and \eqref{FPCM:2})}\end{align}
it follows from \eqref{FPCM:7} that
\begin{align}
\langle x^{**}_0, x^*_0\rangle\geq 0.\end{align} Thus \eqref
{FPCM:0a1} holds and hence $A^*$ is monotone. This establishes
\ref{MaT:3} as required.
\end{proof}

\begin{remark}When $A$ is linear and continuous,
Theorem~\ref{TypeD:1} is due to
 Bauschke and Borwein \cite[Theorem 4.1]{BB}.
 Phelps  and Simons in \cite[Theorem~6.7]{PheSim}  considered the case when
    $A$ is linear but possibly discontinuous; they arrived at
    some of the implications of Theorem 3.1 in that case.

\begin{enumerate}
\item
The proof of \ref{MaT:2}$\Rightarrow$\ref{MaT:3}
in Theorem~\ref{TypeD:1}
 follows closely that of \cite[Theorem~2]{Brezis-Browder}.
\item
Theorem~\ref{TypeD:1}\ref{MaT:3}$\Rightarrow$\ref{MaT:1}
gives an affirmative answer
 to a problem posed by Phelps and Simons
 in \cite[Section~9, item~2]{PheSim}
on the converse of \cite[Theorem~6.7(c)$\Rightarrow$(f)]{PheSim}.
\item
Theorem~\ref{TypeD:1}\ref{MaT:4}$\Rightarrow$\ref{MaT:2}
gives an affirmative answer
 to a problem posed by Simons in \cite[Problem~47.6]{Si2}.
 \item
The    proof of \ref{MaT:3}$\Rightarrow$\ref{MaT:2}
in Theorem~\ref{TypeD:1}
was partially inspired by that of \cite[Theorem~32.L]{Zeidler}
and that of \cite[Theorem~2.1]{MSV}.
\item
The   proof of \ref{MaT:4}$\Rightarrow$\ref{MaT:3}  in Theorem~\ref{TypeD:1}
 closely follows that of \cite[Theorem~4.1(iv)$\Rightarrow$(v)]{BB}.
 \end{enumerate}
\end{remark}

We conclude with an application of Theorem~\ref{TypeD:1}
to an operator studied previously by Phelps and Simons \cite{PheSim}.

\begin{example}
Suppose that $X=L^1[0,1]$ so that
$X^* =  L^{\infty}[0,1]$,
let \begin{align*}D=\menge{x\in X}{\text{$x$ is absolutely
continuous}, x(0)=0, x'\in  X^*},\end{align*}
and set
\begin{equation*}
A\colon X\To X^*\colon
x\mapsto \begin{cases}
\{x'\}, &\text{if $x\in D$;}\\
\varnothing, &\text{otherwise.}
\end{cases}
\end{equation*}
By  \cite[Example~4.3]{PheSim},
$A$ is an at most single-valued maximal monotone linear relation
with proper dense domain, and $A$ is neither symmetric nor skew.
Moreover,
\begin{align*}\dom A^*=\{z\in X^{**}\mid\text{$z$ is absolutely continuous}, z(1)=0,
{z}'\in X^*\} \subseteq X\end{align*}
$A^* z=-z', \forall z\in\dom A^*$, and $A^*$ is monotone.
Therefore, Theorem~\ref{TypeD:1} implies
that $A$ is  of type of (D), of type (NI), and of type (FP).
\end{example}

\section*{Acknowledgments}
Heinz Bauschke was partially supported by the Natural Sciences and
Engineering Research Council of Canada and
by the Canada Research Chair Program.
Jonathan  Borwein was partially supported by the Australian Research  Council.
Xianfu Wang was partially supported by the Natural
Sciences and Engineering Research Council of Canada.


\small


\begin{thebibliography}{99}




\bibitem{BB}
H.H.\ Bauschke and J.M.\ Borwein,
``Maximal monotonicity of dense type, local maximal monotonicity,
and monotonicity of the conjugate are all the same for continuous
linear operators'',
\emph{Pacific Journal of Mathematics},
vol.~189, pp.~1--20, 1999.

\bibitem{BBW}
H.H.\ Bauschke, J.M.\ Borwein, and X.\ Wang,
``Fitzpatrick functions and continuous linear
monotone operators'',
\emph{SIAM Journal on Optimization},
vol.~18, pp.~789--809, 2007.


\bibitem{BC2011}
H.H.\ Bauschke and P.L.\ Combettes,
\emph{Convex Analysis and Monotone Operator Theory in Hilbert Spaces},
Springer-Verlag, 2011.



\bibitem{BWY2}
H.H.\ Bauschke, X.\ Wang, and L.\ Yao,
``Autoconjugate representers for linear monotone
operators'',
\emph{Mathematical Programming (Series B)},
vol.~123, pp.~5--24, 2010.


\bibitem{BWY3}
H.H.\ Bauschke, X.\ Wang, and L.\ Yao,
``Monotone linear relations: maximality
and Fitzpatrick functions'',
\emph{Journal of Convex Analysis}, vol.~16, pp.~673--686, 2009.

\bibitem{BWY4}
H.H.\ Bauschke, X.\ Wang, and L.\ Yao,
``An answer to S.\ Simons' question
on the maximal monotonicity of the sum of a
maximal monotone linear operator and a normal cone operator'',
\emph{Set-Valued and Variational Analysis}, vol.~17, pp.~195--201, 2009.




\bibitem{BWY7}
H.H.\ Bauschke, X.\ Wang, and L.\ Yao,
``Examples of discontinuous
maximal monotone linear operators
and the solution to a recent problem posed by B.F.~Svaiter'',
\emph{Journal of Mathematical Analysis and Applications},
 vol.~370, pp. 224-241,
 2010.


\bibitem{BWY8}
H.H.\ Bauschke, X.\ Wang, and L.\ Yao,
``On Borwein-Wiersma Decompositions of monotone linear
relations'', \emph{SIAM Journal on Optimization}, vol.~20, pp.~2636--2652,
 2010.



\bibitem{BWY9}
H.H.\ Bauschke, X.\ Wang, and L.\ Yao,
``On the maximal monotonicity of the sum  of a maximal monotone
linear relation and the subdifferential operator
of a sublinear function'', to appear \emph{Proceedings of the Haifa Workshop on
 Optimization Theory and Related Topics.
Contemp. Math., Amer. Math. Soc., Providence, RI};\\
\texttt{http://arxiv.org/abs/1001.0257v1}, January  2010.

\bibitem{Bor1}
J.M.\ Borwein,
``Maximal monotonicity via convex analysis'',
\emph{Journal of Convex Analysis}, vol.~13, pp.~561--586, 2006.

\bibitem{Bor2}J.M.\ Borwein, ``Maximality of sums of two maximal monotone operators in general
Banach space'',
\emph{Proceedings of the
 American Mathematical Society}, vol.~135, pp.~3917--3924, 2007.

\bibitem{Bor3}J.M.\ Borwein, ``Fifty years of maximal monotonicity'',
\emph{Optimization Letters}, vol.~4, pp.~473--490, 2010.



\bibitem{Borwein2}
 J.M.\  Borwein, ``A note on $\varepsilon$-subgradients and maximal monotonicity'', \emph{Pacific Journal of Mathematics},
 vol.~103, pp.~307--314, 1982.



\bibitem{BorVan}
J.M.\ Borwein and J.D.\ Vanderwerff,
\emph{Convex Functions},
Cambridge University Press, 2010.


\bibitem{Brezis-Browder}
H.\ Br\'{e}zis and F.E.\ Browder,
``Linear maximal monotone operators
and singular nonlinear integral equations of Hammerstein
type'', in \emph{Nonlinear Analysis (collection of papers in honor of
Erich H.\ Rothe)}, Academic Press, pp.~31--42, 1978.


\bibitem{BuSv1}
 O.\ Bueno and  B.F.\ Svaiter
``A maximal monotone operator of type (D) which maximal monotone extension to the bidual is not of type (D) '',\\
\texttt{http://arxiv.org/abs/1103.0545v1}, March 2011.




\bibitem{BuSv2}
 O.\ Bueno and  B.F.\ Svaiter
``A non-type (D) operator in $c_0$'',\\
\texttt{http://arxiv.org/abs/1103.2349v1}, March 2011.

\bibitem{BurIus}
R.S.\ Burachik and A.N.\ Iusem,
\emph{Set-Valued Mappings and Enlargements of Monotone Operators},
Springer-Verlag, 2008.


\bibitem{Cross}
R.\ Cross, \emph{Multivalued Linear Operators},
Marcel Dekker, Inc, New York, 1998.

\bibitem{FP}
S.\ Fitzpatrick
and R.R.\  Phelps, ``Bounded approximants to monotone operators on Banach spaces'',
 \emph{ Annales de l'Institut Henri Poincar\'{e}. Analyse Non Lin\'{e}aire},  vol.~9, pp.~573--595, 1992.



\bibitem{Gossez3}
J.-P.\ Gossez, ``Op\'{e}rateurs monotones non lin\'{e}aires dans les espaces de Banach non r\'{e}flexifs'',
\emph{Journal of Mathematical Analysis and Applications}, vol.~34, pp.~371--395,  1971.


 \bibitem{MSV}M.\ Marques Alves and B.F.\ Svaiter,
``A new proof for maximal monotonicity of subdifferential operators'',
\emph{Journal of Convex Analysis}, vol.~15, pp.~345--348, 2008.

\bibitem{MarSva2}
M.\  Marques Alves and B.F.\ Svaiter,
``A new old class of maximal monotone operators'',
\emph{Journal of Convex Analysis},
vol.~16, pp.~881--890, 2009.


\bibitem{MarSva}
M.\  Marques Alves and B.F.\ Svaiter,
``On Gossez type (D)
maximal monotone operators'',
\emph{Journal of Convex Analysis},
vol.~17, pp.~1077--1088, 2010.


\bibitem{MLegSva}
J.-E.\  Mart\'{i}nez-Legaz and B.F.\ Svaiter,
``Monotone operators representable by l.s.c. convex functions'',
\emph{Set-Valued  Analysis},
vol.~13, pp.~21--46, 2005.




\bibitem{ph}
R.R.\ Phelps,
\emph{Convex Functions, Monotone Operators and
Differentiability},
2nd Edition, Springer-Verlag, 1993.

\bibitem{PheSim}
R.R.\ Phelps and S.\ Simons, ``Unbounded linear monotone
operators on nonreflexive Banach spaces'',
\emph{Journal of Convex Analysis}, vol.~5, pp.~303--328, 1998.




\bibitem{Rock66}
R.T.\ Rockafellar,
``Extension of Fenchel's duality theorem for
convex functions'',
\emph{Duke Mathematical Journal}, vol.~33, pp.~81--89, 1966.


\bibitem{RockWets}
R.T.\ Rockafellar and R.J-B Wets,
\emph{Variational Analysis}, 3rd Printing,
Springer-Verlag, 2009.

\bibitem{SiNI}
S.\  Simons,
``The range of a monotone operator'',
\emph{Journal of Mathematical Analysis and Applications}, vol.~199, pp.~176--201,  1996.



\bibitem{Si}
S.\  Simons,
\emph{Minimax and Monotonicity},
Springer-Verlag, 1998.


\bibitem{Si4}S.\  Simons,
``Five kinds of maximal monotonicity'',
\emph{Set-Valued and Variational Analysis},
vol.~9, pp.~391--409, 2001.



\bibitem{Si2}
S.\ Simons, \emph{From Hahn-Banach to Monotonicity},
% Lecture Notes in Mathematics, Vol. 1693,
Springer-Verlag, 2008.


\bibitem{Si3}
S.\ Simons,
``A Br\'{e}zis-Browder theorem for SSDB spaces'';\\
\texttt{http://arxiv.org/abs/1004.4251v3}, September 2010.


\bibitem{Svaiter}
B.F.\ Svaiter,
``Non-enlargeable operators and self-cancelling operators'',
\emph{Journal of Convex Analysis},
vol.~17, pp.~309--320, 2010.

\bibitem{Voisei06b}
M.D.\ Voisei,
``A maximality theorem for the sum of maximal monotone
operators in non-reflexive Banach spaces'',
\emph{Mathematical Sciences Research Journal},
vol.~10, pp.~36--41, 2006.



\bibitem{Voisei06}
M.D.\ Voisei,
``The sum theorem for linear maximal monotone operators'',
\emph{Mathematical Sciences Research Journal}, vol.~10, pp.~83--85, 2006.

\bibitem{VZ}
M.D.\ Voisei and C.\ Z{\u{a}}linescu,
``Linear monotone subspaces of locally
convex spaces'',
\emph{Set-Valued and Variational Analysis},
vol.~18, pp.~29--55, 2010.



\bibitem{WY}
 X.\ Wang and L.\ Yao,``Maximally monotone linear subspace extensions of monotone subspaces:
Explicit constructions and characterizations'', submitted;\\
\texttt{http://arxiv.org/abs/1103.1409v1}, March 2011.

\bibitem{Yao}
L.\ Yao,
``The Br\'{e}zis-Browder Theorem revisited and properties of Fitzpatrick functions of order $n$'',
to appear \emph{Fixed Point Theory for Inverse Problems
in Science and Engineering (Banff 2009) , Springer-Verlag};\\
\texttt{http://arxiv.org/abs/0905.4056v1}, May 2009.



\bibitem{Yao2}
L.\ Yao,  ``The sum of a maximally monotone
 linear relation and the subdifferential of a proper lower semicontinuous
  convex function is maximally monotone'', submitted;\\
\texttt{http://arxiv.org/abs/1010.4346v1}, October 2010.


\bibitem{Zalinescu}
{C.\ Z\u{a}linescu},
\emph{Convex Analysis in General Vector Spaces}, World Scientific
Publishing, 2002.

\bibitem{Zeidler}{E.\ Zeidler},
\emph{Nonlinear Functional  Analysis and its Application,
Vol~II/B Nonlinear Monotone Operators},
Springer-Verlag, New York-Berlin-Heidelberg (1990).

\end{thebibliography}
\end{document}